\documentclass[11pt]{amsart}
\usepackage{amssymb,graphicx}
\usepackage{amsfonts}
\usepackage{amsmath}
\usepackage{amsthm}
\usepackage{fancyhdr}
\usepackage{indentfirst}
\usepackage{hyperref}
\usepackage{comment}
\usepackage{color}
\usepackage{subcaption}
\usepackage{epsfig}
\usepackage{pst-grad} % For gradients
\usepackage{pst-plot} % For axes
\usepackage[space]{grffile} % For spaces in paths
\usepackage{etoolbox} % For spaces in paths

\usepackage{cancel}

\newcommand{\V}{\mathcal{V}}
\newcommand{\RV}{\mathcal{RV}}

\newcommand{\R}{\mathbb{R}}

\newcommand{\lan}{\langle}
\newcommand{\ran}{\rangle}

\newcommand{\lc}{\llcorner}
\newcommand{\si}{\sigma}
\newcommand{\Si}{\Sigma}

\newcommand{\ep}{\epsilon}

\newcommand{\Ga}{\Gamma}

\newcommand{\ti}{\tilde}

\newcommand{\M}{\mathbf{M}}
\newcommand{\Gr}{\mathbf{G}}

\newcommand{\bV}{\mathbf{V}}
\newcommand{\bRV}{\mathbf{RV}}

\newcommand{\tM}{\widetilde{M}}

\newcommand{\spt}{\operatorname{spt}}
\newcommand{\dist}{\operatorname{dist}}
\newcommand{\Div}{\operatorname{div}}
\newcommand{\tr}{\operatorname{tr}}

\newcommand{\interior}{\operatorname{int}}

\newcommand{\Clos}{\operatorname{Clos}}
\newcommand{\Span}{\operatorname{span}}

\newcommand{\Mat}[1]{{\begin{bmatrix}#1\end{bmatrix}}}
\newcommand{\rom}[1]{\expandafter\romannumeral #1}
\begin{document}

\newtheorem{theorem}{Theorem}[section]
\newtheorem{proposition}[theorem]{Proposition}
\newtheorem{corollary}[theorem]{Corollary}

\newtheorem{claim}{Claim}

\theoremstyle{remark}
\newtheorem{remark}[theorem]{Remark}

\theoremstyle{definition}
\newtheorem{definition}[theorem]{Definition}

\theoremstyle{plain}
\newtheorem{lemma}[theorem]{Lemma}

\numberwithin{equation}{section}

%Page head and Page foot
%\pagestyle{headings}
%\lhead{} \chead{min-max minimal surfaces} \rhead{}
%\lfoot{} \cfoot{\thepage} \rfoot{}
%\renewcommand{\headrulewidth}{0.4pt}

%date 
\date{\today}
\subjclass{53C42 (primary) 49Q20 (secondary)}

%Tile and Author
\title[Maximum principle with free boundary]{A maximum principle for free boundary minimal varieties of arbitrary codimension}
\author[Martin Li]{Martin Man-chun Li}
\address{Department of Mathematics, The Chinese University of Hong Kong, Shatin, N.T., Hong Kong}
\email{martinli@math.cuhk.edu.hk}

\author[Xin Zhou]{Xin Zhou}
\address{Department of Mathematics, University of California Santa Barbara, Santa Barbara, CA 93106, USA}
\email{zhou@math.ucsb.edu}
\maketitle

\pdfbookmark[0]{}{beg}

%Begin abstract
%\renewcommand{\abstractname}{}    % clear the title
%\renewcommand{\absnamepos}{empty} % originally center
\begin{abstract}
We establish a boundary maximum principle for free boundary minimal submanifolds in a Riemannian manifold with boundary, in any dimension and codimension. Our result holds more generally in the context of varifolds.
\end{abstract}

%%%%%%%%%%%%%%%%%%%%%%%%%%%%%%%%%%
% Section 1   	 		  Introduction	                      		     %
%%%%%%%%%%%%%%%%%%%%%%%%%%%%%%%%%%
\section{Introduction}
\label{S:intro}

Let $N^*$ be a smooth $(n+1)$-dimensional Riemannian manifold with smooth boundary $\partial N^* \neq \emptyset$, whose inward unit normal (relative to $N^*$) is denoted by $\nu_{\partial N^*}$. The metric and the Levi-Civita connection on $N^*$ will be denoted by $\langle \cdot, \cdot \rangle$ and $\nabla$ respectively.

Suppose $N \subset N^*$ is a compact domain whose topological boundary $S:=\partial N$ is a smooth \emph{properly} embedded hypersurface in $N^*$. In other words, $S$ is smooth embedded hypersurface with boundary $\partial S=S \cap \partial N^*$. Let $\nu_S$ denote the unit normal of $S=\partial N$ pointing into $N$. Recall that a point $p \in S$ is said to be \emph{strongly $m$-convex} provided that 
\[ \kappa_1 + \kappa_2 + \cdots + \kappa_m >0 \]
where $\kappa_1 \leq \kappa_2 \leq \cdots \leq \kappa_n$ are the principal curvatures \footnote{The principal curvatures of a hypersurface $S$ (possibly with smooth boundary) at a point $p \in S$ are defined to be the eigenvalues of the second fundamental form $A^S$ as a self-adjoint operator on $T_p S$ given by $A^S(u):=-\nabla_u \nu_S$ where $\nu_S$ is a fixed unit normal to $S$.} of $S$ at $p$ with respect to $\nu_S$. We will often denote $T:=N \cap \partial N^*$ with inward unit normal $\nu_T$ pointing into $N$. Note that any of the hypersurfaces $S, T$ and their common boundary $S \cap T$ could be disconnected. See Figure \ref{proper-domain}.

%\begin{definition}
%\label{D:proper}
%A compact subset $N \subset N^*$ is called a \emph{proper sub-domain of $N^*$} if $N$ is itself an $(n+1)$-dimensional Riemannian manifold with piecewise smooth boundary $\partial N= S \cup T$ where $T=N \cap \partial N^*$ and $S=\overline{\partial N \setminus T}$ \footnote{Throughout this paper, we use $\overline{A}$ to denote the closure of any subset $A \subset N^*$.} are smooth hypersurfaces in $N^*$ with smooth common boundary $S \cap T$. 
%\end{definition}

%Note that any of the hypersurfaces $S, T$ and their common boundary $S \cap T$ could be disconnected. We will denote $\nu_S$ and $\nu_T$ to be the unit normal to $S$ and $T$ respectively that points into $N$. See Figure \ref{proper-domain}. We also regard $N^*$ as a proper sub-domain of itself with $S=\emptyset$ and $T=\partial N^*$, provided that $N^*$ is compact.

\begin{figure}[h]
    \centering
    \includegraphics[width=0.4\textwidth]{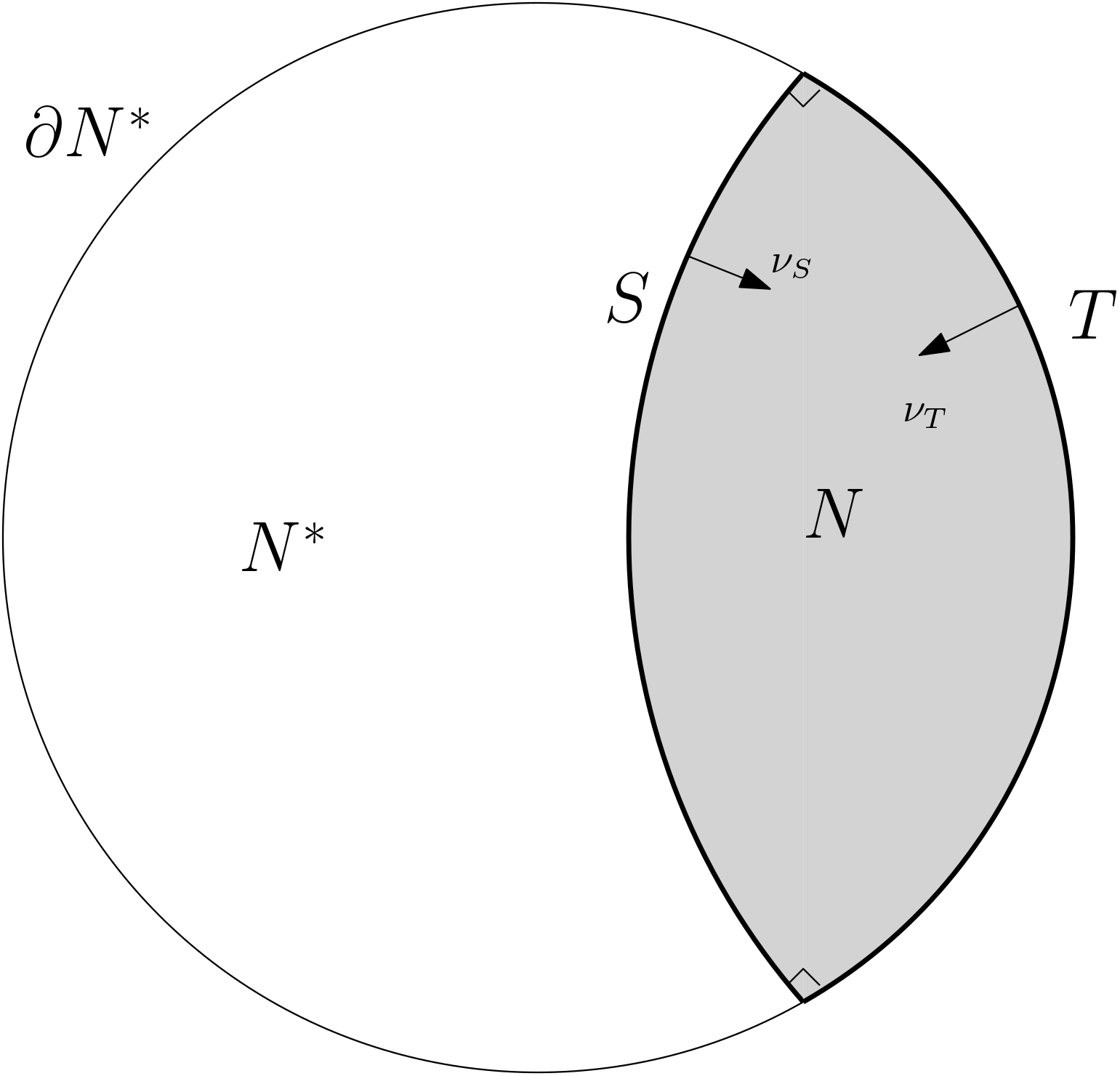}
        \caption{A compact domain $N \subset N^*$ whose topological boundary $S:=\partial N$ is a properly embedded hypersurface meeting $\partial N^*$ orthogonally.}
        \label{proper-domain}
\end{figure}

%\begin{definition}
%\label{D:orthogonal}
%A proper sub-domain $N$ of $N^*$ is said to be
%\begin{itemize}
%\item[(i)] \emph{orthogonal} if $S$ and $T$ intersect each other orthogonally along their common boundary $S \cap T$;
%\item[(ii)] \emph{strongly $m$-convex} at a point $p \in S$ provided that 
%\[ \kappa_1 + \kappa_2 + \cdots + \kappa_m >0 \]
%where $\kappa_1 \leq \kappa_2 \leq \cdots \leq \kappa_n$ are the principal curvatures \footnote{The principal curvatures of a hypersurface $S$ (possibly with smooth boundary) at a point $p \in S$ are defined to be the eigenvalues of the second fundamental form $A^S$ as a self-adjoint operator on $T_p S$ given by $A^S(u):=-\nabla_u \nu_S$ where $\nu_S$ is a fixed unit normal to $S$.} of $S$ at $p$ with respect to $\nu_S$.
%\end{definition}

Consider the following space of ``tangential'' vector fields
\[ \mathfrak{X}(N^*):=\left\{ \begin{array}{c} \text{compactly supported $C^1$ vector field $X$ on $N^*$}\\
\text{such that $\langle X, \nu_{\partial N^*} \rangle =0$ along $\partial N^*$} \end{array} \right\},\]
any $X \in \mathfrak{X}(N^*)$ generates a one-parameter family of diffeomorphisms $\{\phi_t\}_{t \in \mathbb{R}}$ of $N^*$ such that $\phi_0$ is the identity map of $N^*$ and $\phi_t(\partial N^*)=\partial N^*$ for all $t$. If $V$ is a $C^1$ submanifold of $N^*$ with boundary $\partial V \subset \partial N^*$ such that $V$ has locally finite area, then we denote the first variation of area of $V$ with respect to $X$ by:
\begin{equation}
\label{E:1st variation}
\delta V (X) :=\left. \frac{d}{dt} \right|_{t=0} \text{area}(\phi_t(V)).
\end{equation}
Note that (\ref{E:1st variation}) makes sense even when $V$ has infinite total area as the vector field $X$ (hence $\phi_t$) is compactly supported. In fact, the same discussion holds for any varifold $V$. We refer the readers to the appendix of \cite{White09} for a quick introduction to varifolds. We will be following the notations in \cite{White09} closely. Readers who are not familiar with the notion of varifolds may simply replace any varifold $V$ by a $C^1$ submanifold with boundary lying inside $\partial N^*$. 

\begin{definition}
An $m$-dimensional varifold $V$ is said to be \emph{stationary with free boundary} if $\delta V(X)=0$ for all $X \in \mathfrak{X}(N^*)$.
\end{definition}

Note that any $C^1$ submanifold $M$ of $N^*$ with boundary $\partial M= M \cap \partial N^*$ is stationary with free boundary if and only if $M$ is a minimal submanifold in $N$ meeting $\partial N^*$ orthogonally along $\partial M$. These are commonly called \emph{properly embedded} \footnote{See \cite{Li15} and \cite{LiZ16} for a more detailed discussion on properness.} \emph{free boundary minimal submanifolds}.

The goal of this paper is to prove the following result, which generalize the main result of \cite{White10} to the free boundary setting.

\begin{theorem}[Boundary maximum principle for stationary varifolds with free boundary]
\label{T:main-thm}
Let $N \subset N^*$ be a compact domain whose topological boundary $S:=\partial N$ is a properly embedded hypersurface meeting $\partial N^*$ orthogonally. Suppose $S$ is strongly $m$-convex at a point $p \in \partial S$. Then, $p$ is not contained in the support of any $m$-dimensional varifold $V$ which is supported in $N$ and stationary with free boundary.
\end{theorem}

Theorem 1 of \cite{White10} establishes the maximum principle at any \emph{interior} point of $S$ which is strongly $m$-convex. Our result above shows that any stationary varifold with free boundary cannot touch $S$ from inside of $N$ at a strongly $m$-convex point on the \emph{boundary} of $S$ either. In case the varifold $V$ is a $C^2$ hypersurface (i.e. $m=n$) with free boundary lying inside $T$, our theorem follows from the classical boundary Hopf lemma \cite[Lemma 3.4]{Gilbarg-Trudinger} as follows. Suppose $p$ is a boundary point \footnote{Note that $p$ cannot be an interior point. Otherwise, $V$ would have non-empty support outside $N$ by transversality.} of the $C^2$ hypersurface $V$. Using the Fermi coordinate system relative to $T$ centered at $p$ (see \cite[Section 7]{Li-Zhou-A} for example), one can locally express $S$ and $V$ as graphs of functions $f_S$ and $f_V$ respectively over an $n$-dimensional half-ball $B^+_{r_0}=\{x_1^2 + \cdots +x_n^2 <r_0, x_1 \geq 0, x_{n+1}=0\}$ such that $f_V\geq f_S$ because $V$ lies completely on one side of $S$. Then, the difference $u:=f_V-f_S$ is a $C^2$ function on $B^+_{r_0}$ satisfying $Lu \leq 0$ in the interior of $B^+_{r_0}$ for some uniformly elliptic second order differential operator $L$. Moreover, since $S$ is orthogonal to $T$ and $V$ is a free boundary hypersurface, the function $u$ satisfies the following homogeneous Neumann boundary condition along $\{x_1=0\}$:
\begin{equation}
\label{E:free-bdy}
\frac{\partial u}{\partial x_1}=0.
\end{equation}
Since $u \geq 0$ everywhere in $B^+_{r_0}$ and attains zero as its minimum value at the origin, (\ref{E:free-bdy}) violates the boundary Hopf lemma \cite[Lemma 3.4]{Gilbarg-Trudinger}. Our main theorem (Theorem \ref{T:main-thm}) shows that the same result holds in any codimension and in the context of varifolds as well.

The interior maximum principle for minimal submanifolds \emph{without boundary} has been proved in various context. The case for $C^2$ hypersurfaces follows directly from Hopf's classical interior maximum principle \cite[Theorem 3.5]{Gilbarg-Trudinger}. Jorge and Tomi \cite{Jorge-Tomi-03} generalized the result to $C^2$ submanifolds in any codimension. Later, White \cite{White10} proved that the maximum principle holds in the context of varifolds, which has important consequences as for example in the Almgren-Pitts min-max theory on the existence and regularity of minimal hypersurfaces in Riemannian manifolds (see \cite[Proposition 2.5]{Pitts} for example). Similarly, our boundary maximum principle (Theorem \ref{T:main-thm}) is a key ingredient in the regularity part of the min-max theory for free boundary minimal hypersurfaces in compact Riemannian manifolds with non-empty boundary, which is developed in \cite{LiZ16} by the authors. We expect to see more applications of Theorem \ref{T:main-thm} to other situations related to the study of free boundary minimal submanifolds.

Our method of proof of Theorem \ref{T:main-thm} is mostly inspired by the arguments in \cite{White10} (also in \cite[Proposition 2.5]{Pitts}). The key point is to construct a suitable test vector field $X$ which is compactly supported locally near the point $p$ and is universally area-decreasing for any varifold $V$ contained inside $N$ (see \cite[Theorem 2]{White10}). However, the situation is somewhat trickier in the free boundary setting as the test vector field $X$ constructed has to be \emph{tangential}, i.e. $X \in \mathfrak{X}(N^*)$. In the interior setting of \cite{White10}, the vector field $X$ is constructed as the gradient of the distance function from a perturbed hypersurface which touches the boundary of $N$ up to second order at $p$. Unfortunately, the distance function from a free boundary hypersurface does not behave well near the free boundary for at least two reaons. First of all, the distance function may fail to be $C^2$ near the boundary. Second, even if it is smooth, its gradient may not be tangential and thus cannot be used as a test vector field. We overcome these difficulties by constructing a pair of mutually orthogonal foliations near $p$, one of which consists of free boundary hypersurfaces for each leaf of the foliation. We then define our test vector field $X$ to be the unit normal to the foliation consisting of free boundary hypersurfaces and show that it is universally area-decreasing as in \cite{White10}. 
We would like to point out that the same argument also applies to varifolds which only \emph{minimize area to first order in $N$} in the sense of \cite{White10} and to free boundary varieties with bounded mean curvature in a weak sense. 
%The proofs are exactly the same and the results are stated at the end of the paper.

%The key ingredient of the proof is to construct certain local mean convex foliation near a given point $p\in\partial_{rel}K\cap \partial M$. The foliation is chosen so that each slice intersects $\interior(K)$ in a compact subset. Nevertheless, our situation is subtle as we need to make sure that each slice in the foliation meets $\partial M$ orthogonally, or equivalently each slice should have free boundary. We will choose a test variational vector field as the unit normal vector field of this foliation multiplied with some cutoff function. Finally the foliation should satisfy that the unit normal vector field is parallel, at least at the base point $p$. 

%\vspace{1em}
The paper is organized as follows. In Section \ref{S:foliations}, we give a detailed local construction (Lemma \ref{L:foliation}) of orthogonal foliations near a boundary point $p \in \partial N^*$ where a hypersurface $S$ meets $\partial N^*$ orthogonally. We can then choose a local orthonormal frame adapted to such foliation which gives a nice decomposition of the second fundamental form (Lemma \ref{L:frame}). We give the proof of our main result (Theorem \ref{T:main-thm}) in Section \ref{S:proof}. 
%Finally, we state the corresponding generalizations to varifolds which minimize area up to first order and with bounded mean curvature in a weak sense. 
%In section \ref{S:definitions}, we give some basic definitions about varifolds and establish a few preliminary results including a decomposition of the second fundamental form using an orthonormal frame adapted to a free boundary minimal hypersurface. We then prove our main result Theorem \ref{T:max-principle-B} in section \ref{S:max-principle}.
All functions and hypersurfaces are assumed to be smooth (i.e. $C^\infty$) unless otherwise stated.

%\vspace{1em}
%The organization of the paper is as follows. In section \ref{S:definitions} we give various definitions concerning varifolds in Riemannian manifolds with boundary. In section \ref{S:max-principle}, we give the proof of our main result Theorem \ref{T:max-principle-B}.

\vspace{1em}
{\bf Acknowledgements}: 
The authors would like to thank Prof. Richard Schoen for his continuous encouragement. They also want to thank Prof. Shing Tung Yau, Prof. Tobias Colding and Prof. Bill Minicozzi for their interest in this work.
% people to acknowledge by Martin
% people to acknowledge by Xin
The first author is partially supported by a research grant from the Research Grants Council of the Hong Kong Special Administrative Region, China [Project No.: CUHK 24305115]. 
%and CUHK Direct Grant [Project Code: 4053118].
The second author is partially supported by NSF grant DMS-1704393.

%%%%%%%%%%%%%%%%%%%%%%%%%%%%%%%%%%
% Section 2    	 		  Orthogonal Foliations      	             %
%%%%%%%%%%%%%%%%%%%%%%%%%%%%%%%%%%
\section{Orthogonal Foliations}
\label{S:foliations}

Throughout this section, let $N^*$ be an $(n+1)$-dimensional Riemannian manifold with boundary $\partial N^* \neq \emptyset$ as in Section \ref{S:intro}. Let $p \in \partial N^*$ be a point on the boundary of $N^*$. Suppose $S$ is a hypersurface in $N^*$ which meets $\partial N^*$ orthogonally along its boundary $\partial S = S \cap \partial N^*$ containing the point $p$. We first show that one can extend $S$ and $\partial N^*$ locally near $p$ to foliations whose leaves are mutually orthogonal to each other.

\begin{lemma}
\label{L:foliation}
There exists a constant $\delta>0$, a neighborhood $U \subset N^*$ containing $p$ and foliations \footnote{See for example \cite{Lawson74} for a precise definition of a foliation. When $U$ possess a boundary, one requires one of the following: (i) all the leaves are transversal to the boundary; or (ii) every leaf is either contained in the boundary or is completely disjoint from it.} $\{S_s\}, \{T_t\}$, with $s \in (-\delta,\delta)$ and $t\in [0, \delta)$, of $U$ such that $S_0=S \cap U$, $T_0=\partial N^* \cap U$; and $S_s$ intersect $T_t$ orthogonally for every $s$ and $t$. In particular, each hypersurface $S_s$ meets $\partial N^*$ orthogonally. (See Figure \ref{foliation}.)
\end{lemma}

\begin{proof}
We first extend $S$ locally near $p$ to a foliation $\{S_s\}$ such that each $S_s$ meets $\partial N^*$ orthogonally. This can be done in a rather straightforward manner as follows. Let $(x_1,\cdots,x_{n+1})$ be a local Fermi coordinate system of $N^*$ centered at $p$ such that $x_1=\dist_{N^*}(\cdot,\partial N^*)$. Furthermore, we can assume that $(x_2,\cdots,x_{n+1})$ is a local Fermi coordinate system of $\partial N^*$ relative to the hypersurface $S \cap \partial N^*$, i.e. $x_{n+1}$ is the signed distance in $\partial N^*$ from $S \cap \partial N^*$. As in Section \ref{S:intro} we can express $S$ in such local coordinates as the graph $x_{n+1}=f(x_1,\cdots,x_n)$ of a function $f$ defined on a half ball $B^+_{r_0}$ such that $f=0=\frac{\partial f}{\partial x_1}$ along $B^+_{r_0} \cap \{x_1=0\}$. The translated graphs $x_{n+1}=f(x_1,\cdots,x_n)+s$ then gives a local foliation $\{S_s\}$ near $p$ such that each leaf $S_s$ is a hypersurface in $N^*$ which meets $\partial N^*$ orthogonally along its boundary $\partial S_s=S_s \cap \partial N^*$. Note that $\partial S_s$ gives a local foliation of $\partial N^*$ near $p$ obtained from the equi-distant hypersurfaces of $\partial S \subset \partial N^*$.

Next, we construct another foliation $\{T_t\}$ which is orthogonal to \emph{every} leaf of the foliation $\{S_s\}$ defined above. Let $q \in N^*$ be a point near $p$ which lies on the leaf $S_s$. We define $\nu(q)$ to be a unit vector normal to the hypersurface $S_s$. By a continuous choice of $\nu$ it gives a smooth unit vector field in a neighborhood of $p$ such that $\nu(q) \in T_q \partial N^*$ for each $q \in \partial N^*$ since each $S_s$ meets $\partial N^*$ orthogonally. As $\nu$ is nowhere vanishing near $p$, the integral curves of $\nu$ gives a local $1$-dimensional foliation of $N^*$ near $p$. We can put together these integral curves to form our desired foliation $\{T_t\}$ as follows. Let $\Gamma_t \subset S$ be the parallel hypersurface in $S$ which is of distance $t>0$ away from $S \cap \partial N^*$ (measured with respect to the intrinsic distance in $S$). Define $T_t$ to be the union of all the integral curves of $\nu$ which passes through $\Gamma_t$. It is clear that $\{T_t\}$ gives a local foliation near $p$. Since $\nu(q)$ is tangent to the leave $T_t$ which contains $q$, the leaves $S_s$ and $T_t$ must be orthogonal to each other for \emph{every} $s$ and $t$. This proves the lemma.
\end{proof}

\begin{figure}[h]
    \centering
    \includegraphics[width=0.6\textwidth]{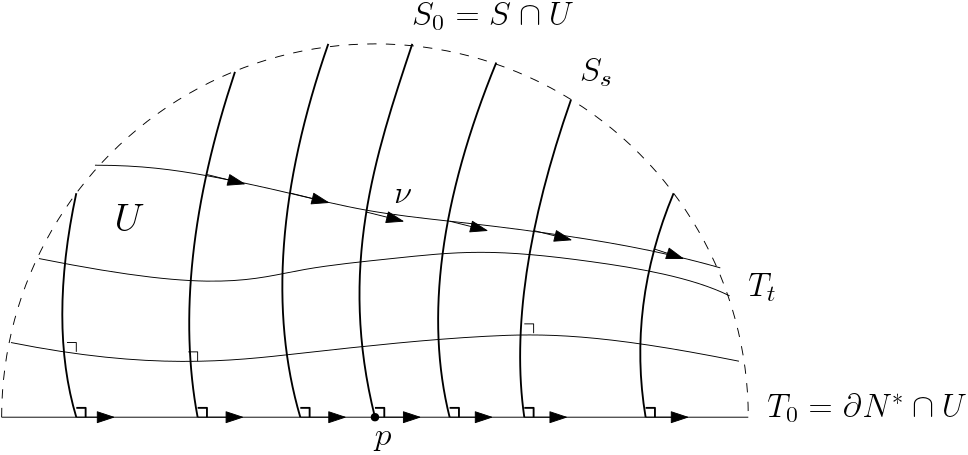}
        \caption{A local orthogonal foliation near a boundary point $p \in S \cap \partial N^*$.}
        \label{foliation}
\end{figure}

Next, we make use of the local orthogonal foliation in Lemma \ref{L:foliation} to give a decomposition of the second fundamental form of the leaves of $\{S_s\}$ under a suitable orthonormal frame.

\begin{lemma}
\label{L:frame}
Let $\{e_1,\cdots,e_{n+1}\}$ be a local orthonormal frame of $N^*$ near $p$ such that at each $q \in S_s \cap T_t$, $e_1(q)$ and $e_{n+1}(q)$ is normal to $S_s \cap T_t$ inside $S_s$ and $T_t$ respectively. Then, we have $\langle A^{S_s}(e_1),e_i \rangle=- \langle A^{T_t}(e_i),e_{n+1} \rangle$ for each $i=2,\cdots,n$, where $A^{S_s}$ and $A^{T_t}$ are the second fundamental forms of the hypersurfaces $S_s$ and $T_t$ in $N^*$ with respect to the unit normals $e_{n+1}$ and $e_1$ respectively.
\end{lemma}

\begin{proof}
By definition of $A^{S_s}$ and $A^{T_t}$ (see Section \ref{S:intro}), we have 
\[ \langle A^{S_s}(e_1),e_i \rangle=\langle -\nabla_{e_1} e_{n+1},e_i \rangle =\langle e_{n+1}, \nabla_{e_i} e_1 \rangle =-\langle A^{T_t}(e_i),e_{n+1} \rangle,\]
where we used the fact that $[e_1,e_i]$ is tangent to $S_s$ in the second equality.
\end{proof}

%Although the \emph{global} existence of a codimension-one foliation is topologically obstructed by vanishing of its Euler characteristic \cite{Thurston76}, it is very easy to construct \emph{local} foliations by studying family of local solutions to some non-singular system of differential equations. In particular, one can integrate a non-vanishing vector field to obtain a local $1$-dimensional foliation. This can be used to show that any $(S,T)$ as above which meets orthogonally near some point $p \in S \cap T$ can be extended to a local orthogonal foliation near $p$. 

%%%%%%%%%%%%%%%%%%%%%%%%%%%%%%%%%%
% Section 3    	 		  Proof of Theorem 1.4        	             %
%%%%%%%%%%%%%%%%%%%%%%%%%%%%%%%%%%
\section{Proof of Theorem \ref{T:main-thm}}
\label{S:proof}
The proof is by a contradiction argument motivated by some of the ideas in \cite{White10}. Note that we will continue to adopt the notations in Section \ref{S:intro}. Suppose on the contrary that there exists a point $p \in \partial S=S \cap T$ which lies in the support of an $m$-dimensional varifold $V$ in $N$ which is stationary with free boundary. Our goal is to construct a tangential vector field $X \in \mathfrak{X}(N^*)$ which is compactly supported near $p$ such that $\delta V(X) <0$ (recall (\ref{E:1st variation})), which contradicts the stationarity of $V$.

%Since $N$ is strongly $m$-convex at $p$ (recall Definition \ref{D:orthogonal} (ii)) and $S$ is orthogonal to $\partial N^*$, it is clear that $N \cap \partial N^*$ is also strongly $m$-convex at $p$ as a sub-domain in $\partial N^*$. 
For every $\epsilon>0$ small, we can define
\[ \Gamma:=\{x \in \partial N^*: \dist_{\partial N^*}(x, \partial S)=\epsilon \dist^4_{\partial N^*}(x,p)\}, \]
which is an $(n-1)$-dimensional hypersurface in $\partial N^*$ that is smooth in a neighborhood of $p$. 

\textit{Claim 1: $\Gamma$ touches $\partial S$ from outside $T$ up to second order at $p$.}

\textit{Proof of Claim 1:} Let $(y_1,\cdots,y_{n-1},t)$ be the Fermi coordinate system of $\partial N^*$ centered at $p$ adapted to the hypersurface $\partial S$, i.e. $(y_1,\cdots,y_{n-1})$ is the geodesic normal coordinates of $\partial S$ centered at $p$ and $t$ is the signed distance function from $\partial S$ in $\partial N^*$ (taken to be negative in $T$). In such a Fermi coordinate system, locally near $p$ we have $\partial S=\{t=0\}$, $T=\{t \geq 0\}$ and 
\[ \Gamma=\{  t=\epsilon g(y_1,\cdots,y_{n-1},t)\}\]
for some function $g$ which is smooth defined near the origin. Using the definition of Fermi coordinates, we have $g$ vanishes up to second order at the origin since the metric components $g_{ij}$ of $\partial N^*$ in Fermi coordinates has $C^2$ bound only in terms of the geometry of $\partial S$ and $\partial N^*$. This proves the claim.

Next we want to extend $\Gamma$ to a hypersurface $S'$ in $N^*$ which meets $\partial N^*$ orthogonally along $\Gamma$ such that $S'$ touches $N$ from outside at $p$ up to second order. The construction of such an $S'$ can be done locally as follows. As in the proof of Lemma \ref{L:foliation}, let $(x_1,\cdots,x_{n+1})$ be a Fermi coordinate system around $p$ such that 
\begin{itemize}
\item $\{x_1 \geq 0\} \subset N^*$, 
\item $\{ x_{n+1} = f(x_1,\cdots,x_n)\} \subset S$, 
\item $\{ x_{n+1} \geq f(x_1,\cdots,x_n)\} \subset N$,
\item $\{x_1=x_{n+1}=0\} \subset \Gamma$.
\end{itemize}
%Note that for $\epsilon$ small, one can express $\Gamma$ as some graph $\{x_{n+1}=g(x_2,\cdots,x_n)\}$ inside $\{x_1=0\}$. 
%Since $\Gamma$ touches $S \cap \partial N^*$ from outside $N \cap \partial N^*$, 
By Claim 1, we have $f(0,x_2,\cdots,x_n) \geq 0$ with equality holds only at the origin. Take $S'$ to be the graph $x_{n+1}=u(x_1,\cdots,x_n)$ of the smooth function 
\[ u(x_1,\cdots,x_n):=\frac{x_1^2}{2} \frac{\partial^2 f}{\partial^2 x_1}(0)  + \frac{x_1^3}{6} \left(\frac{\partial^3 f}{\partial^3 x_1}(0)- \epsilon \right).\]
Since $u=\frac{\partial u}{\partial x_1}=0$ along $\{x_1=0\}$, $S'$ is indeed an extension of $\Gamma$ meeting $\partial N^*$ orthogonally. It is clear from the definition that the Hessian of $u$ and $f$ agrees at the origin. For $\epsilon$ sufficiently small, $f \geq u$ everywhere in a neighborhood of $p$ with equality holds only at the origin where $f$ and $u$ agrees up to second order. In order words, $S'$ touches $N$ from outside up to second order at $p$. 

Since $S'$ meets $\partial N^*$ orthogonally, we can apply all the results in Section \ref{S:foliations} to $S'$ to obtain local foliations $\{S'_s\}$ and $\{T_t\}$ as in Lemma \ref{L:foliation}. We will use the same notations as in the proof of Lemma \ref{L:foliation} in what follows (with $S$ replaced by $S'$). Define a smooth function $s$ in a neighborhood of $p$ such that $s(q)$ is the unique $s$ such that $q \in S'_s$. 

\begin{lemma}
\label{L:grad-s}
$\nabla s =\psi \nu$ for some function $\psi$ which is smooth in a neighborhood of $p$ such that $\psi=1$ along $\partial N^*$.
\end{lemma} 

\begin{proof}
Since $s$ is constant on each leaf $S'_s$ by definition, $\nabla s$ is normal to the hypersurface $S'_s$ and thus $\nabla s= \psi \nu$ for some smooth function $\psi$ in a neighborhood of $p$. The last assertion follows from our construction that $\partial S'_s$ are parallel hypersurfaces from $\partial S'$ in $\partial N^*$.
\end{proof}

Now, we define a vector field $X$ on $N^*$ by
\[ X(q):=\phi(s(q)) \nu(q),\] 
where $\phi(s)$ is the cutoff function defined by
\[
\phi(s)=\left\{ \begin{array}{ll}
e^{1/(s-\ep)} & \text{if $0\leq s<\ep$},\\
0 & \textrm{if $s\geq \ep$}.
\end{array}\right.
\]
As $S'$ touches $N$ at $p$ from outside, we see that $X$ is compactly supported in a neighborhood of $p$. Moreover, since $\nu(q) \in T_q \partial N^*$ at all points $q \in \partial N^*$, we have $X \in \mathfrak{X}(N^*)$. To finish the proof, we just have to show that $X$ decreases the area of $V$ up to first order, i.e. $\delta V(X) <0$. 

At each $q$ in a neighborhood of $p$, we consider the bilinear form on $T_q N^*$ defined by
\[ Q(u,v):= \langle \nabla_u X, v \rangle (q).\]
Let $\{e_1,\cdots,e_{n+1}\}$ be an orthonormal frame as in Lemma \ref{L:frame} (note that $e_{n+1}=\nu$). By Lemma \ref{L:grad-s}, when $u=e_i$, $v=e_j$, $i,j=1,\cdots,n$, we have
\[ Q(e_i,e_j)= \langle \nabla_{e_i} (\phi \nu), e_j \rangle = -\phi \langle A^{S'_s}(e_i),e_j \rangle . \]
Moreover, since $\langle \nu,\nu \rangle \equiv 1$ and $\nabla_{e_i}s \equiv 0$, we have for $i=1,\cdots,n$,
\[ Q(e_i,e_{n+1})=  \langle \nabla_{e_i} (\phi \nu), e_{n+1} \rangle = \phi \langle \nabla_{e_i} \nu, \nu \rangle =0.\]
On the other hand, when $u=e_{n+1}=\nu$, we have 
\[ Q(e_{n+1},e_1)= \langle \nabla_{e_{n+1}} (\phi \nu), e_1 \rangle = \phi \langle \nabla_\nu \nu, e_1 \rangle = \phi \langle A^{T_t}(\nu),\nu \rangle. \]
Since $\langle \nu, e_j \rangle \equiv 0$, we have for $j=2,\cdots,n$,
\[ Q(e_{n+1},e_j)= \langle \nabla_{e_{n+1}} (\phi \nu), e_j \rangle = \phi \langle \nabla_\nu \nu, e_j \rangle. \]
Finally, when $u=v=e_{n+1}=\nu$, using Lemma \ref{L:grad-s} and $\langle \nu,\nu \rangle \equiv 1$,
\[  Q(e_{n+1},e_{n+1})= \langle  (\nabla_\nu \phi) \nu,\nu \rangle + \phi \langle \nabla_\nu \nu,\nu \rangle=  \phi' \psi. \]
Therefore, we can express $Q$ in this frame as the following $n+1$ by $n+1$ matrix: 
\begin{equation}
\label{E:Q}
Q=\Mat{- \phi A^{S'_s}_{11} & \phi A^{T_t}_{n+1,j} & 0  \\
             \phi  A^{T_t}_{i,n+1} & -\phi A^{S'_s}_{ij} & 0 \\ 
              \phi  A^{T_t}_{n+1,n+1} & \phi \langle \nabla_\nu \nu, e_j \rangle & \phi' \psi}
\end{equation}
where $i,j=2,\cdots,n$, and $q \in S_s \cap T_t$.

\begin{lemma}
\label{L:Q}
When $\ep>0$ is small enough, $\tr_P Q<0$ for all $m$-dimensional subspace $P \subset T_q N^*$.
\end{lemma}

\begin{proof}
If $P \subset T_q S'_s$, then $\tr_P Q < 0$ since $S'_s$ is strongly $m$-convex in a neighborhood of $p$. Therefore, we focus on the case $P \not \subset T_q S'_s$. In this case, one can fix an orthonormal basis $\{v_1,\cdots, v_m\}$ for $P$ such that $\{v_1,\cdots,v_{m-1}\} \subset T_q (S'_s \cap T_t)$. As $P \not \subset T_q S'_s$, there exists some unit vector $v_0 \in T_q S'_s$ with $v_0 \perp v_i$ for $i=1,\cdots,m-1$ and $\theta \in (0,\pi)$ such that
\[ v_m= (\cos \theta) \; v_0 + (\sin \theta) \; e_{n+1}. \]
Denote $P'=\Span \{v_0,v_1,\cdots,v_{m-1}\} \subset T_q S'_s$. 
On the other hand, since $v_0 \in T_q S'_s$, one can write 
\[ v_0 =a_1 e_1 + \cdots + a_n e_n, \]
where $a_1^2 + \cdots + a_n^2=1$. 
Therefore, using (\ref{E:Q}) and that $\phi' \leq -\frac{1}{\ep^2} \phi$, by possibly shrinking the neighborhood of $p$ we have $\psi  \geq 1/2$, $|A^{T_t}| \leq K$, $|A^{S'_s}| \leq K$ and $|\langle \nabla_\nu \nu, e_j \rangle| \leq K$ for some constant $K>0$ (depending on the chosen orthogonal foliation in Lemma \ref{L:foliation} but independent of $\epsilon$), one then obtains
\begin{displaymath}
\begin{split}
\tr_P Q &= \sum_{i=1}^{m-1} Q(v_i,v_i) + Q(v_m,v_m) \\
&=  \sum_{i=1}^{m-1} Q(v_i,v_i) + \cos^2 \theta \; Q(v_0,v_0) + \sin \theta \cos \theta \; Q(e_{n+1},v_0) \\
& \hspace{3cm} + \sin^2 \theta \; Q(e_{n+1},e_{n+1}) \\
&= - \phi \tr_{P'} A^{S_s'} +\sin^2 \theta \; \left(\phi' \psi+\phi A^{S_s'}(v_0,v_0) \right) + a_1 \phi \sin \theta \cos \theta \; A^{T_t}_{n+1,n+1}\\
& \hspace{3cm} + \sum_{j=2}^n a_j \phi \sin \theta \cos \theta \; \langle \nabla_\nu \nu, e_j \rangle  \\
&\leq  - \phi \tr_{P'} A^{S_s'} + \phi \left( (K-\frac{1}{2\ep^2})  \sin^2 \theta + \sqrt{n} K |\sin \theta \cos \theta| \right) 
\end{split}
\end{displaymath}

\begin{lemma}
\label{L:calculus}
As $\ep \to 0$, we have
\[   \max_{\theta \in [0,\pi]}\left[  (K-\frac{1}{2\ep^2})  \sin^2 \theta + \sqrt{n} K |\sin \theta \cos \theta| \right] \to 0. \]
\end{lemma}
\begin{proof}
Define the function $F:[0,\pi] \to \mathbb{R}$ by
\[ F(\theta):=   (K-\frac{1}{2\ep^2})  \sin^2 \theta + \sqrt{n} K |\sin \theta \cos \theta|.\]
Notice that $F(\theta)=F(\pi-\theta)$ for all $\theta \in [0,\pi/2]$ and that $F(0)=0$, $F(\pi/2)=K-\frac{1}{2}\ep^{-2}$ which is negative as long as $\epsilon < 1/\sqrt{2K}$. Moreover, if $F'(\theta_0)=0$ at some $\theta_0 \in (0,\pi/2)$, then we have
\begin{equation}
\label{E:theta_0}
\tan 2 \theta_0 = \frac{\sqrt{n} K}{\frac{1}{2}\ep^{-2} -K}.
\end{equation}
Note that such a $\theta_0$ is unique and $\theta_0 \to 0$ as $\ep \to 0$. Using (\ref{E:theta_0}) and L'Hospital's rule, $F(\theta_0) \to 0$ as $\ep \to 0$. This proves Lemma \ref{L:calculus}.
\end{proof}

Using Lemma \ref{L:calculus} and that $S'_s$ is strongly $m$-convex in a small neighborhood of $p$ when $\ep$ is sufficiently small, we have $\tr_P Q < 0$ and thus finished the proof of Lemma \ref{L:Q}.
\end{proof}

%The rest of the proof of Theorem \ref{T:main-thm} follows from the first variation formula for varifolds \cite[\S 39]{Simon} and we refer the reader to \cite{White10} for the details.

To finish the proof, recall that the first variation formula \cite[39.2]{Simon} says
\begin{eqnarray*}
\delta V(X) &=& \int \operatorname{div}_P X(q) \; dV(q,S) \\
&=& \int \tr_P Q (q) \; dV(q,S) < 0,
\end{eqnarray*} 
where the last inequality follows from Lemma \ref{L:Q} and that the support of the vector field $X$ inside $N$ is contained in a very small neighborhood of $p$. This gives a contradiction to the assumption that $V$ is stationary with free boundary. This finishes the proof of Theorem \ref{T:main-thm}.

\bibliographystyle{amsplain}
\bibliography{references-max-principle}

\end{document}